\documentclass[11pt]{article}

\usepackage[english]{babel}

\usepackage{amsthm, amsmath, amssymb, amsbsy, amsfonts, latexsym, euscript}

\usepackage{graphicx}

\usepackage{xcolor}

\newtheorem{theorem}{\bf Theorem}[section]

\newtheorem{lemma}{\bf Lemma}[section]

\newtheorem{corollary}{\bf Corollary}[section]

\theoremstyle{remark}

\newtheorem{remark}{\bf Remark}[section]

\newtheorem{definition}{\bf Definition}[section]

\newtheorem{example}{\bf Example}[section]

\newcommand\ord{\operatorname{ord}}

\begin{document}
\title{On solvability of linear differential equations in finite terms}

\author{Askold Khovanskii\thanks{The first author was partially supported by the Canadian Grant No. 156833-17.}, and Aaron Tronsgard\thanks{The second author was partially supported by the NSERC post-graduate scholarship (PGS-D) 569475}}

\maketitle

To the memory of Andrei Bolibrukh

\begin{abstract} We consider the problem of solvability of linear differential equations over a differential field~$K$. We introduce a class of special differential field extensions, which widely generalizes the classical class of  extensions of differential fields by integrals and by exponentials of integrals and which has similar properties. We announce the following result: if a linear differential equation over $K$ can not be solved by generalized quadratures, then no special extension can help solve it. In the paper we prove a weaker version of this result in which we consider only pure transcendental extensions of $K$. Our paper is self-contained and understandable for beginners. It demonstrates the  power of Liouville's original approach to problems of solvability of equations in finite terms.
\end{abstract}

\begin{center}
	
	2010 Math. Subj. Class. 12H20.

\end{center}

\section{Introduction}\label{sec1}

In the 1830s Liouville started  to create a theory of solvability in finite terms.
In some cases it answers the following question: is there a solution of an equation which is representable by a certain class of formulas? In other words, is there a solution representable in {\it finite terms}?

About the same time Galois theory was invented. It provides a criterion of solvability of algebraic equations in radicals.  Liouville was inspired by Galois theory. Among his results there are two famous theorems. The First Liouville Theorem provides a criterion of integrability of elementary functions in elementary functions. The Second Liouville Theorem provides a criterion of solvability of second order linear differential equations by quadratures. Modern proofs of these theorems can be found correspondingly in \cite{Kho19} and \cite{Kho18}.

Liouville's brilliant and simple ideas were never properly understood (maybe, because Liouville did not carefully formalize them).  Only in 1910, Picard and Vessiot generalized the Second Liouville Theorem  for linear differential equations of an arbitrary order. To obtain this result they developed a new deep differential Galois theory.

J.F. Ritt up to some extent  formalized Liouville's method and generalized it (see \cite{Ritt} and \cite{kho}). Later J.F. Ritt, Kolchin and others developed differential algebra, which totally gets rid of multi-valued analytic functions and deals with abstract differential fields. It became a well established branch of pure algebra.

If one wants to show that a solution of an equation is not representable by any formula of a certain type (for example is not representable by quadratures), one has to deal with the collection of all functions representable by such formulas. Functions from this collection could be very complicated; they could be multi-valued, they could have complicated singularities, and so on. Such functions do not form a differential field, since arithmetic operations are not well defined for multi-valued functions.

In the framework of differential algebra, one can state and solve the problem of solvability of linear differential equations by quadratures. These results are  applicable to analytic theory of differential equations in the following way.

Let $K$ be a subfield  of the field $M$ of meromorphic functions on a connected open set $U$ on a complex line, which is closed under differentiation. Consider a homogeneous linear differential equation over the field $K$
\begin{equation}\label{anal}
y^{(n)}+a_1y^{n-1}+\dots+ a_n y = 0,
\end{equation}
where $a_i\in K$. With a point $z\in U$ which is a regular point for all coefficients $a_i$, one can associate: the differential field $K_z$ of germs at $z$ of functions from  $K$, the differential field $F_z$ generated over $K_z$ by germs at $z$ of all solutions of (\ref{anal}) and the differential field $M_z$ of all meromorphic germs  at~$z$.

\begin{definition}\label{def} An equation (\ref{anal}) is solvable by quadratures  over $K$ if there exists a point $z\in U$, such that inside the field $M_z$ there is a chain of extensions by quadratures of the field $K_z$ which contains the field $F_z$.
\end{definition}

One can show that the definition of solvability of an equation by quadratures which makes use of multi-valued analytic functions  is equivalent to Definition \ref{def} (see \cite{Kho18}, \cite{top}).

Thus, results from differential algebra are applicable to the analytic theory of differential equations. Moreover, in differential algebra one can consider equations over arbitrary differential fields, not only over a field of meromorphic functions.

On the other hand, there are operations over analytic functions which do not exist in differential algebra. For example, one can compose functions, and it is natural to use compositions in formulas. Problems on solvability of equations by formulas involving compositions can not be stated and solved using a pure algebraic approach.
Note, that compositions with some functions can be defined in differential algebra (see remark \ref{remark2.1}).

A reduction of the Second Liouville Theorem (and its generalizations for linear differential equations of arbitrary order) to differential algebra should, in principal, significantly simplify it: it means that the result can be proven using only arithmetic operations and differentiation. Nevertheless, algebraic results on solvability of linear differential equations of arbitrary order, which uses differential Galois theory (see \cite{pic}) and Rosenlicht's algebraic proof (see \cite{rosen} and \cite{rosen2}) of that result, which uses valuation theory, are rather involved. In \cite{Kho18}, the same result over functional differential fields is proved using Liouville's original method. This proof is relatively simple, but it uses analytic tools and is not applicable to abstract differential fields.

Our goal is to find an algebraic proof of that result which is as simple as the original proof of the Second Liouville's Theorem. We are also looking for possible generalizations of that result.

We have achieved our goal. The result is announced in the next section. Currently, one step in our proof which deals with algebraic extensions is not polished and is too involved. We have decided in this first publication, to exclude algebraic extensions from consideration.

This makes our result weaker than the classical criterion of solvability of linear differential equations by quadratures. But also our result in many ways is stronger than the classical criterion: we show that many extensions (which we call special transcendental extensions) can not help to solve linear differential equations. Besides that, our result is understandable for beginners and it clearly demonstrates Liouville's original approach.

In subsequent publications we plan to present a proof of the announced theorem. It is not as elementary as the present paper. We need the theory of algebraic curves over (possibly very big) algebraically closed fields. We also plan to consider much more general special extensions which have similar properties with special transcendental extensions.

\section{Solvability  over differential fields}
	We begin with the definition of an abstract differential field.
\begin{definition}
		A differential field is a field $K$ along with a fixed additive map $K \to K$ sending $a \mapsto a'$ satisfying the Leibniz rule $(ab)' = a'b + ab'$. The element $a'$ is called the derivative of $a \in K$.
	\end{definition}
Let us present  several examples of differential fields.

\begin{example}
\begin{enumerate}
	\hfill
\item The field $M(U)$ of all meromorphic functions on a connected open set $U$ on the extended complex line $\Bbb C^1 \cup \{\infty\}=\Bbb CP^1$ with the natural differentiation.
\item The field of meromorphic functions $M(S)$ on a connected Riemann surface $S$ equipped with non constant meromorphic projection $\pi: S\to \Bbb CP^1$ with the following  natural differentiation: for $f\in M(S)$ the derivative $f'$ is defined as the ratio $\frac{d f} {d \pi}$ of meromorphic differential forms $d f$ and $d \pi$.

\item Any subfield of the fields $M(U)$ and $M(S)$ closed under differentiation.

The differential field $\Bbb C(z)$ of rational functions of complex variable $z$ and the differential field $\Bbb C(S)$ of rational functions on an algebraic curve $S$ with non constant  rational projection $\pi:S\to \Bbb CP^1$ belong to these examples.
\end{enumerate}
\end{example}

An element $c$ of a differential field $K$ is a {\it constant} if $c'=0$. The collection of all constants in $K$ form the {\it field of constants}.

{\it Throughout the paper, by a differential field we mean a differential field of characteristic 0 with the same fixed subfield of constants}. For example, one can assume that all differential fields under consideration contain $\Bbb C$ as the subfield of constants.

Consider a nested pair $K\subset F$ of differential fields.
\begin{definition} An element $y\in F$ is called
\begin{enumerate}
\item {\it algebraic over $K$} if $y$ is a root of a polynomial over $K$;
\item an {\it integral over $K$} (or {\it an integral of $a$ $\in K$})  if $y'=a$  and $a\in K$;
\item an {\it exponential of an integral over $K$} (or  {\it an {exponential of integral of $a$ $\in K$}}) if $y'=a y$ and $a\in K$.
\end{enumerate}
\end{definition}

\begin{remark}\label{remark2.1}
Many other natural definitions can be reformulated in such a way that they will make sense over differential fields. For example, the functions $u=\exp f$ and $v=\ln f$ satisfy the following relations: $u'=f'u$, $v'=f'/f$. Thus, one can say that the element $u$ of $F$ is an {\it exponential of $f\in K$ } if $u'=f'u$, and the element $v$ of $F$ is a {\it logarithm} of $f$ if $v'=f'/f$.

Correspondingly, many different types of representability of functions in finite terms can be reformulated for differential fields.
\end{remark}

Let us define a notion of solvability by quadratures in differential algebra. Consider a nested pair $K \subset F$ of differential fields. Let $E$ be an intermediate differential field, i.e., $K \subset E \subset F$.
\begin{definition}
An extension $K\subset E$ is an {\it extension by quadratures} if there is a chain of field extensions
\begin{equation}\label{chain}
K=K_0\subset\dots\subset K_n=E
\end{equation}
such that for each $0 \leq i<n$, the extension $K_i\subset K_{i+1} $ is either an extension by adjoining an integral over $K_i$, or by adjoining an exponential of an integral over $K_i$.
\begin{enumerate}
\item An element $a\in F$ is {\it representable by quadratures over $K$} if  there is an extension $K\subset E$ by quadratures such that $a\in E$.
\item An equation over $K$ is {\it solvable by quadratures over $K$} if  there is an extension $K\subset E$ by quadratures which contains all solutions of the equation.
    \end{enumerate}
\end{definition}

In the same way one defines other types of solvability of equations in differential algebra.

\begin{definition}
An extension $K\subset E$ is an {\it extension by generalized quadratures} if there is a chain (\ref{chain}) in which besides adjoining integrals and exponential of integrals, algebraic extensions are allowed.
\begin{enumerate}
\item An element $a\in F$ is {\it representable by generalized quadratures over $K$} if  there is an extension $K\subset E$ by generalized quadratures such that $a\in E$.
\item An equation over $K$ is {\it solvable by generalized quadratures over $K$} if there is an extension $K\subset E$ by generalized quadratures which contains all solutions of the equation.
\end{enumerate}
\end{definition}

Each extension $K_i\subset K_{i+1}$ in a chain (\ref{chain}) defining an extension  by quadratures, or by generalized quadratures is generated over $K_i$ by one element $y_i$. The element $y_i$ is either algebraic, or transcendental over $K_i$. If $y_i$ is algebraic, one can consider the extension $K_i\subset K_{i+1} $ as an algebraic equation. Otherwise, $y_i$ has to satisfy an equation $y'_i=\mathcal R(y_i)$, where $\mathcal R$ is a rational function over $K_i$ (see Lemma \ref{principle}).

\begin{definition}\label{special} An element $y$ is a {\it special transcendental element over  $K$} if $y$ is a transcendental element over $K$ and it satisfies a differential equation $y'=\mathcal R(y)$, where $\mathcal R$ is a rational function over $K$ which is divisible by $y$, i.e., is representable in the form $\mathcal R= \frac{yP}{Q}$, where the polynomials $yP$ and $Q$ are relatively prime.
\end{definition}

\begin{lemma} If an  element $y$ is transcendental over $K$ and either

\begin{enumerate}
\item $y$ is an exponential of integral over $K$,

\item $y$ is representable as $y=u^{-1}$, where $u$
is an integral over $K$,
\end{enumerate}
then $y$ is a special transcendental element over $K$.
\end{lemma}
\begin{proof}
	\begin{enumerate}
	\item If $y'=ay$, then $y'=\mathcal R(y)$, where the rational function  $\mathcal R(y)=ay$ is divisible by $y$;
	\item if $u'=a$, then $y'=-ay^2$, i.e., $y'=\mathcal R(y)$, where the rational function  $\mathcal R(y)=-ay^2$ is divisible by $y$.
    \end{enumerate}
\end{proof}
The extension of $K$ by $u$ is identical to its extension by $u^{-1}$.

\begin{definition} An extension $K\subset E$ is {\it admissible} if there is a chain (\ref{chain}) such that each extension $K_i\subset K_{i+1}$ is obtained by adjoining to $K_i$ a special transcendental element over $K_i$.
\end{definition}

\begin{definition} An extension $K\subset E$ is a {\it generalized admissible extension} if there is a chain (\ref{chain}) such that each extension $K_i\subset K_{i+1}$ is either algebraic, or is obtained by adjoining to $K_i$ a special transcendental element over $K_i$.
\end{definition}

We announce the following theorem.

\begin{theorem}[On strong non solvability of linear differential equations]
	\label{an}
	A homogeneous linear differential equation over $K$ can be solved by  generalized admissible extensions if and only if it can be solved by generalized quadratures.

	In other words, if such an equation can not be solved by generalized quadratures, then no special transcendental extensions can help solve it either.
\end{theorem}

In the present paper, we prove a weaker version of Theorem \ref{an} which provides a criterion of solvability of homogeneous linear differential equations over $K$ by admissible extension. It implies the following:

\begin{theorem} If a homogeneous linear differential equation can be solved by admissible extensions, then it can be solved by a nested chain of extensions by integrals and by transcendental exponentials of integrals.
\end{theorem}

\section{An illustrative example}

In this section,  we present a simple example which illustrates  our approach  and  Liouville's method.

Let $K$ be a differential field of meromorphic functions on some connected domain on a complex line which is closed under differentiation.

Consider  the second order homogeneous equation
\begin{equation}
\label{I.2.eq0}
y''+ay'+by=0
\end{equation}
over $K$, i.e.,  $a=a(x)$ and $b=b(x)$ are functions from $K$. It is well known that {\it $y$ is a nonzero solution of (\ref {I.2.eq0}) if and only if its logarithmic derivative $u=y'/y$ satisfies the Riccati equation}
\begin{equation}
\label{I.eq1}
u'+ a u+b+u^2=0.
\end{equation}

Thus, problems of solving in finite terms equations (\ref{I.2.eq0}) and (\ref {I.eq1}) are closely related, and for solvability by quadratures are in fact equivalent. If one wants to show that the equation (\ref{I.2.eq0}) is not solvable by quadratures, one can try to show it instead for equation (\ref{I.eq1}). Let us demonstrate how one can make one step in proving non-solvability by quadratures of the equation (\ref{I.eq1}).

Assume that there are no solutions of the Riccati equation (\ref {I.eq1}) in the differential field $K$. Let us show that an addition to the field $K$, say, an exponent $y $ can not help in solving (\ref {I.eq1}). Our goal is to prove the following theorem.

\begin{theorem}
\label{I.2}
If an exponent $y$ (i.e., a  solution of the equation $y'=y$) is transcendental over $K$, and the Riccati equation (\ref{I.eq1}) has a solution in the extension $K\langle y \rangle$ of the differential field $K$ by $y$ then it has a solution in $K$.
\end{theorem}

\begin{remark} Liouville's theory of integrability in finite terms has the following slogan:

{\it If a ``simple" equation has a solution given by an ``allowed"  formula, then it has to have (another) solution given by a ``simple" allowed formula}.

Theorem \ref{I.2} demonstrates this slogan: here a simple equation is the equation (\ref{I.eq1}), an allowed formula represents a solution as an element of $K\langle y\rangle$ and a simple allowed formula represents a solution as an element of $K$.
\end{remark}

Let us prove the first two lemmas.
\begin{lemma}[Liouville's principle for a transcendental exponent]
\label{lemma5}
An element $a \in K\langle y \rangle$ is representable in a unique way as a rational function in $y$ over $K$. The derivative $R'$ of $R=R(y)$ is given by
\begin{equation}\label{liu}
R' = R'_x + R'_y y,
\end{equation}
where $R'_x=\frac{\partial}{\partial x} R$,  $R'_y=\frac{\partial}{\partial y} R$ and $\frac{\partial}{\partial x}$, $\frac{\partial}{\partial y}$ are the following differentiations    of the field $K(y)$:
\begin{enumerate}
		\item $\frac{\partial}{\partial x}$ sends $y$ to  zero  and  coincides with the differentiation in $K$ on coefficients;
		\item $\frac{\partial}{\partial y}$ sends $y$ to 1 and sends to zero all elements of the differential field $K$.
\end{enumerate}
\end{lemma}

\begin{proof} Indeed, since $y'=y$, any element of the differential field  $K\langle y \rangle$  is representable as a rational function in $y$ over $K$. If an element $a$ can be represented as a rational function in $y$ in two different ways, i.e., $a=R_1(y); \, \, a=R_2(y)$, then $R_1(y)=R_2(y)$ which means that $y$ is algebraic over $K$ and contradicts the assumption. One can check (\ref{liu}) using the chain rule and the identity $y'=y$.
\end{proof}

\begin{remark}\label{rem3.2} If $y$ is algebraic element over $K$ then for different rational functions $R_1, R_2$ the elements $R_1(y)$ and $R_2(y)$ could be equal. For some rational functions $R$ the element $R(y)$ is not defined. Assume that an element $y$ (either transcendental or algebraic over $K$) satisfies the equation $y'=y$ and for some $R$ the element $a=R(y)$ is defined then for its derivative $a'$ the following relation holds $a' = R'_x(y) + y R'_y (y)$.
\end{remark}

According to Lemma \ref{lemma5}, the field $K\langle y\rangle$ is isomorphic to the field $K(y)$ of rational functions over $K$ with the differentiation given by (\ref{liu}). In the field of rational functions $K(y)$, the order of a rational function at the point $y=0$ is defined.

\begin{definition}\label{order} A rational function $R(y)$ over $K$ has order $m$ at the point $y=0$ if it is representable as $R=y^mR_0$, where $R_0$ is a rational function which is not divisible by any positive or negative degree of $y$.
\end{definition}

Using (\ref{liu}), one can check the following lemma.

\begin{lemma}
\label{lemma7}
If a rational function $R$ has order $m$ at the point $y=0$, then its derivative $R'= R'_x+R'_y y$ has order at $y=0$ not smaller than $m$.
\end{lemma}

\begin{proof} If $R=y^mR_0$, then $$R'=my^mR_0 +y^{m+1}(R_0)'_y + y^m (R_0)'_x.$$
 The function $R_0$ equals to $P/Q$, where $Q$ is not divisible by $y$. Thus, $(R_0)'_y$ and $(R_0)'_x$ are equal to $(P'_yQ-PQ'_y)/Q^2$ and $(P'_xQ-PQ'_x)/Q^2$, where polynomial $Q^2$ is not divisible by $y$. So $(R_0)'_y$ and $(R_0)'_x$ have nonnegative order which implies Lemma \ref{lemma7}.
\end{proof}

\begin{proof}[Proof of Theorem \ref{I.2}] Assume that $u=R(y)$ satisfies (\ref{I.eq1}). If we plug in to the rational function $R$ any solution $y$ of the equation $y'=y$, then we still will get a solution of (\ref {I.eq1}) under assumption that the plugging in of $y$ to $R$ is well defined (see Remark \ref{rem3.2}).

Let us try to plug in to $R$ the solution $y\equiv 0$. If we succeed, then we will have a solution $R(0)$ of the Riccati equation (\ref {I.eq1}) in the field $K$. One can do it only if the order of the rational function $R(y)$ at the point $y=0$ is nonnegative. If $R$ has a pole at the point $y=0$, then plugging $y= 0$ in to $R$ makes no sense.

Let us show that if $u=R(y)$ has a negative order $k$ at $y=0$, then it can not satisfy equation (\ref {I.eq1}). Indeed, if $u$ has an order $k<0$, then $u^2$ has an order $2k<0$. The terms $u$ and $u'$ in (\ref {I.eq1}) have orders $\geq k$. So the term $u^2$ can not be cancelled and Riccati's equation can not be satisfied. Thus, if $u=R$ satisfies (\ref {I.eq1}), it can not have negative order at $y=0$.

The theorem is proven.
\end{proof}

\section{Generalized Riccati equation}

In this section, we recall  classical results on the generalized Riccati equation (the presentation  is borrowed from \cite{Kho18}). Let $y$ be a a nonzero element of a differential field and let $u$ be its  logarithmic derivative, i.e., $y'=uy$.

\begin{definition} Let $D_n$ be a polynomial in $u$ and in its derivatives $u', \dots, u^{(n-1)}$ up to order $(n - 1)$ defined inductively by the following conditions:
$$D_0=1; D_{k+1}=\frac{d D_k}{d x} + uD_k.$$
\end{definition}

\begin{lemma}
 \begin{enumerate}
 \item The polynomial $D_n$ has integral coefficients and $\deg D_n =n$. The degree $n$ homogeneous part of $D_n$ equals to $u^n$ (i.e., $D_n =u^n + \tilde D_n$, where $\deg \tilde D_n < n$).
	
\item  If $y$ is a function whose logarithmic derivative is equal to $u$ (i.e., if $y'=uy$), then, for any $n\geq 0$, we have the relation $y^{(n)} = D_n(u)y$.
    \end{enumerate}
\end{lemma}

Both claims of the lemma can be easily checked by induction.

Consider a homogeneous linear differential equation whose coefficients $a_i$ belong to a differential field $K$:
\begin{equation}
\label{eq7}
 y^{(n)}+a_1y^{(n-1)}+\dots +a_n y=0.
\end{equation}

\begin{definition} The equation
\begin{equation}
\label{ric}
D_n + a_1 D_{n-1} +\dots+ a_n D_0 = 0
\end{equation}
of order $(n-1)$ is called the {\it generalized Riccati equation} for the homogeneous linear differential equation
(\ref{eq7}).
\end{definition}

\begin{lemma}
 \label{lemma9}
A nonzero element $y$ satisfies the linear differential equation (\ref{eq7}) if and only if its logarithmic derivative $u=y'/y$ satisfies the corresponding generalized Riccati equation (\ref{ric}).
 \end{lemma}

\begin{proof} For proving Lemma \ref{lemma9} in one direction, one can divide (\ref{eq7}) by $y$ and use the identity $y^{(k)}/y=D_k(u)$.
	
In the other direction, if $u$ is a solution of (\ref{ric}), then multiply (\ref{ric}) by $y$ and use the identity $y^{k} = D_k(u)y$. We obtain that $y$ is a nonzero solution of (\ref{eq7}).
\end{proof}

\section{Reduction of order}

In this section, we recall the classical procedure of order reduction for homogeneous linear differential equations (the presentation  is borrowed from \cite{top}).

\subsection{Division with a remainder of differential operators}

{\em A linear differential operator of order $n$} over a differential field $K$ is an operator $L= a_nD^n+\dots +a_0$, where $a_i\in K$ and $a_n\neq 0$, acting on an element $y$ of any differential field $F$ containing $k$  by the formula
$$
L(y)=a_ny^{(n)}+\dots +a_0y.
$$
For operators $L_1$ and $L_2$ over $K$, their product $L=L_1\circ L_2=L_1(L_2)$ is also an operator over $K$. The multiplication of operators is, in general, noncommutative, but the leading term  of the operator $L=L_1\circ L_2$ is equal to the product of the leading terms of the operators $L_1$ and $L_2$. For operators $L$ and $L_2$ of orders $n$ and $k$ over $K$,
there exist unique operators $L_1$ and $R$ over $K$ such that $L=L_1\circ L_2+R$, and the order of $R$ is strictly less than $k$. The operator $R$ is called {\em the remainder in the right division} of the operator $L$ by the operator $L_2$.

The operators $L_1$ and $R$ can be constructed explicitly: the algorithm for the division of operators with a remainder is based on the formula for the leading term of the product and is very similar to the algorithm of the division with a remainder for polynomials in one variable.

\subsection{Procedure of order reduction}
Let $L$ be a linear differential operator over $K$, $y_1$ a nonzero element of the field $K$, $p=\frac{y_1'}{y_1}$ its logarithmic derivative and $L_2= D-p$ the first order operator annihilating $y_1$.

\begin{lemma} The operator $L$ is divisible by $L_2$ from the right if and only if the element $y_1$ satisfies the identity $L(y_1)\equiv 0$.
\end{lemma}
\begin{proof} The remainder $R$ in the right division of $L$ by $L_2$ is the operator of multiplication by $c_0$, where $c_0=\frac{1}{y_1}L(y_1)$. Indeed, the desired equality can be obtained by plugging in $y=y_1$ to the identity $L(y)\equiv L_1\circ L_2(y)+c_0y$.
\end{proof}

Using a nonzero solution $y_1$ of the equation $L(y)=0$ of order $n$, one can reduce the order of this equation. To this end, one needs to represent the operator $L$ in the form $L=L_1\circ L_2$, where $L_1$ is an operator of order $(n-1)$.

{\em Coefficients of the operator $L_1$ lie in an extension of the differential field $K$ obtained by adjoining the logarithmic derivative $p$ of the element $y_1$}. If one knows some solution $u$ of the equation $L_1(u)=0$, then one can construct a certain solution $y$ of the initial equation $L(y)=0$. To do that, it suffices to solve the equation $L_2(y)=y'-py=u$.

\begin{lemma}\label{simple} An element $y$ satisfies the equation $L_2(y)=u$ if and only if $y = y_1 z$, where $z$ satisfies the equation $z' = u / y_1$.
\end{lemma}

\begin{proof} The operator $L_2(y)=y'-\frac{y_1'}{y_1}y$ can be rewritten in the form $L_2(y)= y_1(\frac{y}{y_1})'$. \end{proof}

The procedure just described is called the
\index{reduction of order}
{\em reduction of order} of a differential equation.

\medskip

\noindent{\sc Remark 1.}
An operator annihilating $y_1$ is defined up to a factor, which can be an arbitrary element of the field $K$, and the procedure of the order reduction depends on the choice of this element. It is easier to divide $L$ from the right by the operator $\tilde L_2=D\circ y_1^{-1}$ and represent $L$ in the form $L=\tilde L_1  \tilde L_2$ which reduces the initial equation $L(y)=0$ to the equation $\tilde L_1(u)=0$ of smaller order. Usually, this very procedure of order reduction is given in textbooks on differential equations. Note that the coefficients of the operator $\tilde L_1$ lie in the extension of the differential field $K$ obtained by adjoining the element $y_1$ itself, rather than its logarithmic derivative $p$. This makes the operator $\tilde L_1$ inconvenient for the purposes of our paper.



\section{Pure transcendental extensions}
A field extension $K\subset F$ is {\it pure transcendental} if any element  $y\in F$, which is algebraic over $K$, belongs to $K$. Let us start with the following example. Let $y$ be an integral over $K$, i.e., $y'=a$ for some $a\in K$.

\begin{lemma}\label{integral}
The integral $y$ either belongs to the field $K$ or is transcendental over $K$.
\end{lemma}
\begin{proof} Assume that an integral $y$ of an element $a\in K$  is algebraic over $K$ but does not belong to $K$.  Let $Q$ be a  polynomial over $K$ of the smallest possible degree $n>1$  such that $Q(y)=a_ny^n\dots +a_0=0$.
We can assume   that $a_n=1$. Differentiating the identity $Q(y)=0$, we obtain the equation
\begin{equation}\label{int}
(na +a'_{n-1} )y^{n-1}+\dots+(a_1a+a'_0)=0
\end{equation}
 The coefficient $n a+a'_{n-1}$ is not equal to zero. Indeed, otherwise the derivative of an element $-\frac{a_{n-1}}{n}$ is equal to $a$. It means that $y$ belongs to $K$, since $y=-\frac{a_{n-1}}{n} +C$ for some constant $C$, and  all constants belong to $K$. The equation (\ref{int}) for $y$ has degree $n-1<n$. The contradiction proves that $y$ is not algebraic over $K$.
\qed\end{proof}

Lemma \ref{integral} implies that any nontrivial extension by adjoining an integral is pure transcendental. Let us describe all pure transcendental  extensions of transcendental degree one.

Let $K\subset F$ be a nested pair of differential fields. Assume that $y\in F$ is transcendental over $K$ and the algebraic field $K(y)\subset F$ generated by $y$ and $K$ is closed under differentiation, in particular, $y'$ belongs to the  algebraic field $K(y)$, i.e., $y'=R(y)$, where $R$ is some rational function over $K$.

Liouville's principle shows that for any $R\in K(y)$ there is a pure transcendental extension $K\subset K\langle y\rangle$ in which the identity $y'=R(y)$ holds and the rational function $R$ together with  differentiation in the field $K$ completely determine the differential field $K\langle y\rangle$.

\begin{lemma} \label{difrat} Let $K$ be a differential field, and $K(t)\supset K$ the field of rational functions over $K$ in indeterminate $t$. Then for any choice of rational function $\mathcal{R} \in K(t)$, there exists a unique differentiation on the field $K(t)\supset K$ which coincides with the differentiation on the subfield $K$ and is equal to $\mathcal R(t)$ on $t$.
\end{lemma}

\begin{proof} Consider the map sending $R\in K(t)$ to $R'\in K(t)$ given by the formula
\begin{equation}
				\label{diff}
R' = R'_K + \mathcal{R} R_t,
\end{equation}
where $R'_K$ denotes the derivative of $R$ by the differentiation rule of $K$, viewing $t$ as a constant. And $R_t$ denotes the usual partial derivative with respect to $t$, viewing the elements of $K$ as constant. One can check the the map $R\to R'$ provides a differentiation of the field $K(t)$ which satisfies all needed conditions. Such differentiation is unique, since  $K(t)$ is generated by $t$ and $K$.
\end{proof}

\begin{lemma}[Liouville's Principle]\label{principle}
Let $F$ be some differential field extension of $K$, and $y \in F$ transcendental over $K$. Suppose that $y$ satisfies an equation $y' = \mathcal{R}(y)$, where $\mathcal{R}$ is a rational function with coefficients in $K$.
Then $K \langle y \rangle$ is isomorphic to the field $K(t)$ of rational functions over $K$ with differentiation given  $t'=\mathcal R(t)$ (see Lemma \ref{difrat})
\end{lemma}
		
\begin{proof} Consider the map from the field $K(t)$ of rational functions over $K$ to the field $K\langle y\rangle$ which fixes elements of $K$ and sends $t$ to $y$. Since $y$ is transcendental over $K$, this map is injective. Indeed, if $R_1(y)=R_2(y)$ for different functions $R_1,R_2\in K(t)$ then the identity	 $R_1(y)=R_2(y)$ provides an algebraic equation over $K$ on $y$ which is impossible.
	
So, the field $K\langle y \rangle$ contains an isomorphic copy of (the algebraic field) $K(t)$, and by Lemma \ref{difrat} $K(t)$ can be made a differential field with the choice $t' = \mathcal{R}(t)$. Our map preserves differentiation between the differential field $K(t)$ and its image in $K \langle y \rangle$. And since $K\langle y \rangle$ is the smallest differential field containing $K$ and $y$, we must have that $K \langle y \rangle$ is precisely the image of $K(t)$.
\end{proof}

The following lemma is obvious.
\begin{lemma} Let $K=K_0\subset \dots\subset K_n=F$ be a chain of fields. If for any $i=0,\dots,n-1$ the extension $K_i\subset K_{i+1}$ is pure transcendental, then the extension $K\subset F$ also is pure transcendental.
\end{lemma}

\begin{proof} Assume that $y\in K_{i+1}$ is algebraic over $K_0$. Then since $K_0\subset K_i$, the element $y$ belongs to $K_i$. By repeating this argument, one obtains that $y\in K_0$.
\end{proof}
	
Consider a homogeneous linear differential equation (\ref{eq7})  over the field $\Bbb C(z)$ of complex rational functions with the usual differentiation.

\begin{lemma}\label{trans} If all coefficients of the equation (\ref{eq7}) are polynomials $a_i\in \Bbb C[z]$, then the extension $M$ of the field $\Bbb C(z)$, obtained by adjoining all solutions of (\ref{eq7}), is pure transcendental.
\end{lemma}

\begin{proof} The extension $M$ is contained in the field of meromorphic functions on the complex line. A meromorphic function is algebraic over $\Bbb C(z)$ if and only if it is rational.
\end{proof}

\section{Adjoining a special transcendental element}

In this section, we prove Lemma \ref{mainlemma} which provides an important property of special transcendental extensions.

For a  rational function   $R\in K(y)$ over a field $K$, we denote by $\ord R$ its order  at the point $y=0$ (see Definition \ref{order}). For any functions  $R_1, R_2 \in K(y)$, the following relations hold:
\begin{enumerate}
\item $\ord (R_1 R_2)=\ord R_1+\ord R_2$
\item $\ord (R_1+R_2) \geq \min (\ord R_1, \ord R_2)$, and if $\ord R_1\neq \ord R_2$, \\
then $\ord (R_1+R_2)= \min (\ord R_1, \ord R_2)$.
\end{enumerate}

Let $D:K(y)\to K(y)$ be any differentiation of the field $K(y)$.

\begin{lemma}\label{group} Nonzero elements $f\in K(y)$ satisfying inequality $\ord D(f)\geq \ord f$ form a multiplicative subgroup in the field $K(y)$.
\end{lemma}
\begin{proof} The inequality $\ord D(f)\geq \ord f$ can be rewritten as $\ord (D f/f)\geq 0$. The logarithmic derivative $Df/f$ provides a homomorphism of the multiplicative group of the field $K(y)$ to its additive group. Element $f$ satisfies $\ord Df\geq \ord f$  if and only if it belongs to the pre-image of the additive subgroup of functions $g\in K(y)$ having nonnegative order.
\end{proof}

Let $y$ be a special transcendental element over a differential field $K$ satisfying the differential equation $y' = \mathcal R (y) = y R_0 (y)$, where $R_0$ has nonnegative order at $y=0$, i.e., $\ord R_0 \geq 0$. The differential field $K\langle y\rangle$ is isomorphic to the field $K(y)$ with the differentiation $R' = R'_K + y R_0 R_y$.

\begin{lemma}\label{mainlemma} For any element $R$  in the differential field $K\langle y\rangle$, the following inequality holds:
\begin{equation}\label{val}
\ord R'\geq \ord R
\end{equation}
\end{lemma}

\begin{proof} Indeed, $R' = R'_K + y R_0 R'_y$. It is easy to check that $\ord R'_K \geq \ord R$ and $\ord y R_y' \geq \ord R$. Indeed, these inequalities obviously  hold for polynomials over $K$. Now one can apply Lemma \ref{group}, and the lemma follows from the properties of order listed  above.
\end{proof}

\section{Existence of a solution in an admissible extension}

In this section, we prove the following theorem.

\begin{theorem}\label{firstmain} A homogeneous linear differential equation (\ref{eq7}) has a solution in some admissible extension of the field $K$ if and only if it has a solution $y$ whose logarithmic derivative belongs to $K$, and $y$ either belongs to $K$, or is transcendental over $K$.
\end{theorem}

The proof is based on the following lemma.

\begin{lemma}\label{ricca} Let $K\subset K\langle y\rangle$ be an extension by adjoining a special transcendental element $y$ over $K$. The generalized Riccati equation (\ref{ric}) for the equation (\ref{eq7}) has a solution in $K\langle y\rangle$ if and only if it has a solution in the field $K$.
\end{lemma}

Let $T$ be a polynomial over $K$ in $u$ and in its derivatives $u, u', \dots, u^{(k)},\dots$.

\begin{definition} Say that $T$ of some degree $n$ is a {\it Rosenlicht type polynomial} if the degree $n$ homogeneous part of $T$ is equal to $u^n$ (i.e., $T =u^n + \tilde T$, where $\deg \tilde T < n$).
	
An equation
\begin{equation}\label{ro}
T(u,u',\dots, u^{(n)})=0,
\end{equation}
where $T$ is a Rosenlicht type polynomial, is called a {\it Rosenlicht type equation}.
\end{definition}

The generalized Riccati equation is a Rosenlicht type equation. This property of the generalized Riccati equation plays the key role in the classical approach (which goes back to Liouville and was specified by Rosenlicht) to the problem of solvability of homogeneous linear differential equations in finite terms. Thus, instead of Lemma \ref{ricca}, we prove the following more general lemma.

\begin{lemma}\label{rosen} Let $K\subset K\langle y\rangle$ be an extension by adjoining a special transcendental element $y$ over $K$. An equation (\ref{ro}) has a solution in $K\langle y\rangle$ if and only if it has a solution in the field $K$.
\end{lemma}

\begin{proof} Let $R(y)\in K\langle y\rangle$ where $y$ is a special transcendental element over $K$. Let us show that if $R(y)$ satisfies (\ref{ro}), then $\ord R=m$ is nonnegative. Indeed, if $m<0$ then $\ord u^n=nm$ is strictly smaller than $\ord \tilde T$, which follows from the definition of Rosenlicht polynomial and Lemma \ref{mainlemma}. Thus, $m\geq 0$. So, one can evaluate both sides of (\ref{ro}) at y=0 and obtain a solution of (\ref{ro}) in the field $K$.
\end{proof}

\begin{corollary} \label{coro} An equation (\ref{ro}) has a solution in an admissible extension $E$ of the field $K$ if and only if it has a solution in the field $K$.
\end{corollary}
\begin{proof} Let $K_0\subset\dots, \subset K_n=E$ be a chain in which each extension $K_i\subset K_{i+1}$ is obtained by adjoining a  special transcendental element over $K_i$. By Lemma \ref{rosen}, if (\ref{ro}) has a solution in $K_{i+1}$, it also has a solution in $K_i$. This argument proves corollary.
\end{proof}

\begin{proof}[Proof of Theorem \ref{firstmain}] By Lemma \ref{lemma9}, the logarithmic derivative of any nonzero solution $y$ of (\ref{eq7}) has to satisfy its generalized Riccati equation. Assume that  $y$ belongs to an admissible extension $E$ of $K$. Corollary \ref{coro} implies that the generalized Riccati equation has a solution $u\in K$. Thus, (\ref{eq7}) has a nonzero solution $z$ such that $z'/z=u\in K$. If this element $z$ is algebraic over $K$, it belongs to $K$, since $E$ is a pure transcendental extension of $K$.
\end{proof}

\section{Algebraic relations between exponentials of integrals}

Let $K$ be a differential field. For $y$ belonging to the multiplicative group $K^*=K\setminus \{0\}$ of the field $K$, the logarithmic derivative $y'/y$ is defined. It is easy to check the following lemma.

\begin{lemma} The map, which sends each element $y\in K^*$ to its logarithmic derivative, provides a homomorphism of the multiplicative group $K^*$ to the additive group $K_+$ of the field $K$, i.e., for any $u,v\in K^*$  any integers $k,n$ the following identity holds:

$$ \frac{(u^kv^n)'}{u^kv^n}=k\frac{u'}{u}+n\frac{v'}{v}.$$
\end{lemma}

\begin{definition} Let $F$ be any extension of a differential field $K$. A nonzero element $y\in F$ is an {\it exponential of integral} of $a\in K$ if $y'=ay$. In other words, $y\in F^*$ is an exponential of integral of $a\in k$ if $a$ is the logarithmic derivative of $y$.
\end{definition}

Let $F$ be an extension of  $K$.

\begin{lemma}\label{monom}  Assume that a nonzero element $y\in F$ satisfies an equation  $y'=ay$ for  $a\in K$ and  an element $u\in F$ is representable in a form $u=c y^m$, where $c\in K$,  $m\in \Bbb Z$. Then
\begin{enumerate}
\item  the derivative $u'$ of $u$ also is representable in the form $u' = \tilde c y^m$, where $\tilde c=c'+ma$;
\item $u'=0$ if and only if $c=\lambda y^{-m}$, where $\lambda$ is a constant, i.e., $\lambda'=0$. If $u'=0$, then $y^m\in K$.
\end{enumerate}
\end{lemma}

\begin{proof} The first relation follows from direct differentiation. If $u'=0$, then $u$ is a constant $\lambda$. Thus, $cy^m=\lambda$ which implies the relation $c=\lambda y^{-m}$ and $y^m=c^{-1}\lambda$.
\end{proof}

\begin{definition} Let $F=K\langle y \rangle$ be an extension of a differential field $K$ by an element $y$. Then $F$ is an {\it extension by exponential of integral of $a\in K$} (a {\it transcendental extension by  exponential of integral of $a\in K$}) if $y'=ay$ (if $y$ is transcendental over $K$ and $y'=ay$).
\end{definition}

\begin{theorem}\label{exp}  Let $y_1,\dots,y_n\in F$ be the exponentials of integrals of elements $a_1,\dots, a_n \in K$, where $K$ is a differential subfield of $F$. Assume that there are no nontrivial  monomials $u=y_1^{k_1}\dots y_n^{k_n}$, where $k_1,\dots, k_n$ are not necessary positive integers (assuming that, at least, one $k_i\neq 0$) such that a relation $u=c$ for some $c\in K$ holds. Then the elements $y_1,\dots,y_n$ are algebraically independent over the field $K$.
\end{theorem}

\begin{proof} Assume that $y_1,\dots,y_n$ are algebraically dependent and let  $P=\sum c_{m_1,\dots,m_n}x_1^{m_1}\dots x_n^{m_n}$ be a nonzero Laurent polynomial  over $K$ (i.e., $m_1,\dots,m_n$ are not necessarily positive integral numbers)  which annihilates $y_1,\dots,y_n$, i.e., $\sum c_{m_1,\dots,m_n}y_1^{m_1}\dots y_n^{m_n}=0$. One can choose such Laurent polynomial $P$ which contains the smallest possible number of nonzero terms $c_{m_1,\dots,m_n}y_1^{m_1}\dots y_n^{m_n}$. By dividing the relation $P=0$ by one of its terms, one obtains a Laurent polynomial $\tilde P$ having the same number of terms as $P$, but one of its terms is equal to $1$. By Lemma \ref{monom}, the derivative $u'$ of term $u=c_{m_1,\dots,m_n}y_1^{m_1}\dots y_n^{m_n}$ has a form $u'=\tilde c_{m_1,\dots,m_n}y_1^{m_1}\dots y_n^{m_n}$ for some $ \tilde c_{m_1,\dots,m_n}\in K$.

Thus, by differentiating the identity $\tilde P=0$, one obtains a relation $\tilde P'(y_1,\dots,y_n)=0$ containing smaller number of terms (the derivative of the term $1$ is equal to zero). Thus, the derivative of each term of the Laurent polynomial $\tilde P$ is equal to zero. By Lemma \ref{monom}, it means that some nontrivial monomial $y_1^{m_1}\dots y_n^{m_n}$ belongs to the field $K$.
\end{proof}

\section{Transcendental extensions by exponentials of integrals}

Assume that a differential field $F=K\langle y_1,\dots,y_n\rangle$  is generated over a field $K$ by nonzero elements  $y_1,\dots, y_n$, such that $y'_i=a_i y_i$, where $a_i$ belong to $K$. Let $G$ be the multiplicative subgroup in $F^*$ generated by elements $y_1,\dots, y_n$, and let $G_0$ be its subgroup which is equal to the intersection of $G$ with the field $K$, $G_0=G\cap K^*$. Denote by $\Psi_0\subset \Psi$ the images of the multiplicative groups $G_0, G$  under the map which sends $y\in F^*$ to $y'/y\in F_+$. By construction, the group $\Psi$ is generated by the elements $a_i\in K_+$.

\begin{theorem}\label{factor} The factor group $G/G_0$ has the following properties:
\begin{enumerate}
\item If $G/G_0$ is torsion free, then there is a chain of fields $K=K_0\subset \dots \subset K_m=F$, such that  for $i=0,\dots, m-1$ the field $K_{i+1}$ is equal to $K_i\langle u_{i+1}\rangle$, where $u_{i+1}$ is an exponent of an integral over $K_i$, transcendental over $K_i$.

\item If $G/G_0$ has nontrivial torsion, then there is no chain of fields $K=K_0 \subset \dots \subset K_m \supset F$, such that for $i=0,\dots, m-1$ the field $K_{i+1}$ is a pure transcendental extension of the field $K_i$.
    \end{enumerate}
\end{theorem}

\begin{proof}
	\begin{enumerate}
\item  Let $u_1,\dots,u_m$ be generators of the factor group $G/G_0$. The elements $u_1,\dots,u_m$ can be chosen to be exponentials of integrals over the field $K_0$. Each monomial in $u_1,\dots,u_m$ does not belong to the subgroup $G_0=G\cap K_0$, since the factor group is torsion free. Thus, by Theorem \ref{exp}, elements $u_1,\dots,u_m$ are algebraically independent over $K_0$.  Let $K_i$ for $i=1,\dots,m$ be the field $K_0\langle u_1,\dots,u_i\rangle$.  By construction in the chain of fields, $K_0\subset K_1\dots\subset K_m$. for $i=0,\dots, m-1$ the field $K_{i+1}$  is the extension of $K_i$ by exponent of integral $u_{i+1}$ over the field $K_0\subset K_i$. The field $K_m$ is equal to the field $F$.  All these extensions $K_i\subset K_{i+1}$ are pure transcendental, since $u_1,\dots,u_m$ are algebraically independent over $K_0$.

\item If the group $G/G_0$ has nontrivial torsion, then there is a monomial $u$ in $y_1,\dots, y_n$ which does not belong to $G_0$ but some positive power of it does. Thus, $u$ does not belong to the field $K_0$  but some positive power of it does.

The monomial $u$ is an algebraic element over $K_0$ which does not belong to $K_0$. Any chain $K_0\subset\dots\subset K_m$ of pure transcendental extensions provides a transcendental extension $K_0 \subset K_m$ which can not contain the element $u\in F$.
	\end{enumerate}
\end{proof}

\section{Criterion of solvability in admissible extensions}		 

Let $L = a_n D^n + \ldots + a_0$, where $a_i\in K$ be a linear differential operator of order $n$ over $K$. In this section, we will state and prove a criterion of solvability of the homogeneous equation
\begin{equation}\label{first}
L(y)= a_ny^{(n)}+\dots + a_1y=0
\end{equation}
over $K$ by an admissible extension $E$  of $K$.

\begin{theorem} All solutions of (\ref{first}) belong to an admissible extension $E \supset K$ if and only if  the following two conditions simultaneously hold:
\begin{enumerate}
\item there is a sequence  of nonzero elements $z_1,\dots, z_n$ which are exponential of integrals over $K$, i.e., $z'_1=p_1 z_1,\dots, z'_n=p_nz_n$ and $p_1,\dots, p_n\in K$
such that the operator $L$ is representable in the form
\begin{equation}\label{decom}
L=a_nL_n\circ \dots\circ L_1,
\end{equation}
where each operator $L_i$ is an operator of  order one. Moreover, $L_i$ is equal to $D - p_i$;

\item let $G$ be the multiplicative group generated by $z_1,\dots,z_n$ and let $G_0$ be the intersection of $G$ with $K$, $G_0=G\cap K$. Then the factor group $G/G_0$ is torsion free.

\end{enumerate}
\end{theorem}

\begin{proof} If there is an admissible extension $E$ containing all solutions of (\ref{decom}), then by Theorem \ref{firstmain}, the equation  has to have a solution $z_1$ whose  logarithmic  derivative $p_1$ belongs to $K$. Let $L_1$ be the operator $y'-p_1y_1$. Let us divide $L$ on $L_1$ from the right, $L=\tilde L_1\circ L_1$. The coefficients of $\tilde L_1$ belongs to $K$, since $p_1\in K$. Each solution $u$ of $\tilde L_1 (y) = 0$ is representable in the form $u=L_1(y)$, where $y$ is a solution of the original equation. So all solutions of $\tilde L_1 (y) = 0$ also belong to the admissible extension $E$ of the field $K$.

Thus, one can apply Theorem \ref{firstmain} to the equation $\tilde L_1 (y) = 0$ to find its solution $z_2$ whose logarithmic derivative $p_2$ belongs to $K$. Let $L_2$ be the operator $D - p_2$ and let us represent $\tilde L_1$ in the form $\tilde L_1=\tilde L_2\circ L_2$.

Repeating  this process, one obtains a sequence of exponentials of integrals of $p_1, \ldots , p_n \in K$ for which decomposition (\ref{decom}) of the operator $L$ holds.

The exponential of integrals $z_1,\dots,z_n$ belong to an admissible extension $E$. By Theorem \ref{factor}, the group $G/G_0$ is torsion free.

The criterion for solvability is proven in one direction. Let us prove it in the opposite direction. If the group $G/G_0$ is torsion free, by Theorem \ref{factor}, one can construct an admissible extension $E_1$ of $K$ which contains all the exponential of integrals $z_1,\dots, z_n$. Over the field $E_1$ the operator $L$ is decomposed as a product of factors $D - p_i$. Thus, the equation $L=0$ one can solved step by step, solving equations $L_i(y)=u$. Since $z_i\in E$, solutions of these equation can be obtained in extensions by integrals (see Lemma \ref{simple}). By Lemma \ref{integral}, such extensions are pure transcendental. So, extending $E_1$ by a chain of extensions obtained by adjoining integrals one at a time, one obtains an admissible extension $E$ which contains all solutions of (\ref{first}).
\end{proof}

Note that if the group $G/G_0$ has nontrivial torsion, then the extension of $K$ by adjoining all solutions of (\ref{first}) is not pure transcendental. If it is pure transcendental (as in Lemma \ref{trans}), then the group $G/G_0$ automatically has no torsion. Such an equation is solvable by admissible extension if and only if the operator $L$ admits a decomposition (\ref{decom}).

\bigskip

Askold Khovanskii\\
Department of Mathematics, University of Toronto, Toronto, ON, Canada\\
email: askold@math.toronto.edu
\bigskip

Aaron Tronsgard\\
Department of Mathematics, University of Toronto, Toronto, ON, Canada\\
email: aaron.tronsgard@mail.utoronto.ca


\begin{thebibliography}{widestlabel}
\bibitem{Kho19} Khovanskii, A. \textit{Integrability in finite terms and actions of Lie
groups}. Moscow Mathematical Journal, 19(2):329–341. 2019.


\bibitem{Kho18} Khovanskii, A. \textit{Solvability of equations by quadratures and Newton's theorem}.  Arnold Mathematical Journal, 4(2):193-211. 2018.

\bibitem{Ritt} Ritt, J.F. \textit{Integration in Finite Terms. Liouville's Theory of Elementary Methods}. Columbia University Press, New York, NY. 1948.

\bibitem{kho} Khovanskii, A. \textit{Comments on JF Ritt's Book ``Integration in Finite Terms''}. in book: Integration in Finite Terms: Fundamental Sources by Clemens Raab (Editor), Michael F. Singer (Editor). Springer. Series: Text \& Monographs in Symbolic Computation. 2022. 137-201.

\bibitem{top} Khovanskii, A. \textit{Topological Galois Theory}. Springer-Verlag, Heidelberg. 2014.
	
\bibitem{pic} Marius van der Put and Michael Singer. \textit{Galois theory of linear differential equations, volume 328 of Grundlehren der Mathematischen Wissenschaften [Fundamental Principles of Mathematical Sciences]}. Springer-Verlag, Berlin. 2003.
	
	\bibitem{rosen} Rosenlicht, M. \textit{An analogue of l'Hospital's rule}. Proc. Amer. Math. Soc. 37:369-373. 1973.
    \bibitem{rosen2} Rosenlicht, M. \textit{Differential extension fields of exponential type}. Pacific journal of mathematics, 57(1):289-300. 1975.
\end{thebibliography}
\end{document}